\documentclass[10pt,reqno]{amsart}

\usepackage{a4wide}
\usepackage[all]{xy}
\usepackage{mathbbol}
\usepackage{mathrsfs}
\usepackage{hyperref}
\usepackage[mathscr]{eucal}

\newtheorem{proposition}{Proposition}[section]
\newtheorem{theorem}[proposition]{Theorem}

\newcommand{\cst}{\ifmmode\mathrm{C}^*\else{$\mathrm{C}^*$}\fi}
\newcommand{\tens}{\otimes}
\newcommand{\id}{\mathrm{id}}
\newcommand{\comp}{\!\circ\!}
\newcommand{\I}{\mathbb{1}}
\newcommand{\ph}{\varphi}
\newcommand{\CC}{\mathbb{C}}
\newcommand{\ZZ}{\mathbb{Z}}
\newcommand{\GG}{\mathbb{G}}
\DeclareMathOperator{\C}{C}
\DeclareMathOperator{\qs}{\mathscr{QS}}
\DeclareMathOperator{\Mor}{Mor}
\DeclareMathOperator{\QMap}{\mathsf{Q}-Map}

\numberwithin{equation}{section}

\begin{document}

\title{On quantum maps into quantum semigroups}

\date{\today}

\author{Piotr M.~So{\l}tan}
\address{Department of Mathematical Methods in Physics, Faculty of Physics, University of Warsaw\newline
\indent{and}\newline
\indent{}Institute of Mathematics, Polish Academy of Sciences}
\email{piotr.soltan@fuw.edu.pl}
\urladdr{http://www.fuw.edu.pl/~psoltan/en/}

\thanks{Partially supported by Polish government grant no.~N201 1770 33, European Union grant PIRSES-GA-2008-230836 and Polish government matching grant no.~1261/7.PR UE/2009/7.}

\begin{abstract}
We analyze the recent examples of quantum semigroups defined by M.M.~Sadr who also brought up several open problems concerning these objects. These are defined as quantum families of maps from finite sets to a fixed compact quantum semigroup. We show that these are special cases of free products of quantum semigroups. This way we can answer all the questions stated by M.M.~Sadr. Along the way we discuss the question whether restricting the comultiplication of a compact quantum group to a unital $\mathrm{C}^*$-subalgebra defines such a structure on the subalgebra. In the last section we show that the quantum family of all maps from a non-classical finite quantum space to a quantum group (even a finite classical group) might not admit any quantum group structure.
\end{abstract}

% \subjclass[2000]{20G42, 46L89}

\maketitle

\section{Introduction}\label{intro}

Let $X$ and $Y$ be locally compact Hausdorff topological spaces. Then the space $\C(X,Y)$ carries a natural topology --- the compact-open topology. As a topological space $\C(X,Y)$ is characterized by the following universal property: let us denote by $\psi$ the continuous map $\C(X,Y)\times{X}\to{Y}$ given by
\[
\psi(f,x)=f(x);
\]
then for any locally compact Hausdorff topological space $M$ and a continuous map $\ph\colon{M}\times{X}\to{Y}$ there is a unique continuous $\lambda\colon{M}\to\C(X,Y)$ such that
\[
\psi\bigl(\lambda(m),x\bigr)=\ph(x)
\]
for all $x\in{X}$ i $m\in{M}$. Note that when $X$ is finite and $Y$ is compact then $\C(X,Y)$ is compact.

The above characterization might be rather artificial for topological spaces. However, when we pass to the category dual to the category of locally compact topological spaces, i.e.~the category of commutative \cst-algebras such a characterization turns out to be very useful. Moreover we can use this description of $\C(X,Y)$ to define an analog of this space when $X$ and $Y$ are no longer \emph{classical} spaces, but \emph{quantum} spaces.

Quantum spaces (called \emph{pseudospaces} in \cite{pseu}) are virtual objects which correspond to noncommutative \cst-algebras in a way analogous to how locally compact spaces correspond to commutative \cst-algebras. For any \cst-algebra $A$ we write $\qs(A)$ for the corresponding quantum space. When $A$ happens to be commutative then we can identify $\qs(A)$ with the unique locally compact space $X$ such that $A\cong\C_0(X)$ --- the space of continuous functions vanishing at infinity on $X$.

In \cite{pseu,qs} quantum analogs of spaces of the form $\C(X,Y)$ for $X$ finite and $Y$ a compact subset of $\CC^n$ were defined and shown to exist. This means that for any finite-dimensional \cst-algebra $B$ and any unital finitely generated \cst-algebra $A$ there exists a unique \cst-algebra $C$ equipped with a unital $*$-homomorphism
\begin{equation}\label{Phi}
\Phi\colon{A}\longrightarrow{B}\tens{C}
\end{equation}
such that for any \cst-algebra $D$ and any $\Psi\in\Mor(A,B\tens{D})$ there exists a unique $\Lambda\in\Mor(C,D)$ such that
\begin{equation}\label{first}
\xymatrix{
A\ar[rr]^-{\Phi}\ar@{=}[d]&&B\tens{C}\ar[d]^{\id\tens\Lambda}\\
A\ar[rr]^-{\Psi}&&B\tens{D}}
\end{equation}
Here $\Mor(\cdot,\cdot)$ is the space of \emph{morphisms of \cst-algebras} as defined e.g.~in \cite[Section 1]{pseu}. The pair $(C,\Phi)$ is already determined uniquely by requiring that for any \emph{unital} \cst-algebra $D$ and a unital $*$-homomorphism $\Psi\colon{A}\to{B}\tens{D}$ there exists a unique unital $*$-homomorphism $\Lambda\colon{C}\to{D}$ making \eqref{first} commute.

Thus given a finite-dimensional \cst-algebra $B$ and a unital finitely generated \cst-algebra $A$ we have a new \cst-algebra $C$ together with a unital $*$-homomorphism \eqref{Phi}. In view of the characterization of the space of continuous maps discussed above it makes sense to call $\qs(C)$ the \emph{quantum space of all maps} from $\qs(B)$ to $\qs(A)$ denoted by $\QMap\bigl(\qs(B),\qs(A)\bigr)$.

In particular one can consider the case $A=B$. Then the quantum space $\QMap\bigr(\qs(B),\qs(B)\bigr)$ corresponding to the \cst-algebra $C$ (denoted simply by $\QMap\bigr(\qs(B)\bigr)$) carries a natural structure of a \emph{compact quantum semigroup}, i.e.~ there exists a unital $*$-homomorphism $\Delta_C\colon{C}\to{C}\tens{C}$ which is \emph{coassociative}:
\[
(\Delta_C\tens\id)\comp\Delta_C=(\id\tens\Delta_C)\comp\Delta_C.
\]
(\cite[Section 4]{qs}). Existence of this structure generalizes the fact that given a finite set $X$ the set of all maps $X\to{X}$ is a compact (in fact finite) semigroup. Let us emphasize that in the category of quantum spaces this construction leads to highly nontrivial quantum spaces (cf.~Example \ref{example}).

In \cite{mms} M.M.~Sadr analyzed the analog of another natural phenomenon, namely that the space of all maps from a finite set to a compact semigroup is in itself a compact semigroup. He showed that a non-commutative version of this fact is true. Namely, if $\qs(A)$ carries a structure of a quantum semigroup (so that we have a coassociative $\Delta_A\colon{A}\to{A}\tens{A}$) and $B$ is commutative then there exists a unique $\Gamma\colon{C}\to{C}\tens{C}$ such that the diagram
\begin{equation}\label{Gam}
\xymatrix
{
A\ar[rr]^-\Phi\ar[d]_{\Delta_A}&&B\tens{C}\ar[d]^{\id\tens\Gamma}\\
A\tens{A}\ar[d]_{\Phi\tens\Phi}&&B\tens{C}\tens{C}\\
B\tens{C}\tens{B}\tens{C}\ar[rr]_{\id\tens\chi\tens\id}&&B\tens{B}\tens{C}\tens{B}\ar[u]_{\mu\tens\id\tens\id}
}
\end{equation}
(where $\chi$ is the flip $C\tens{B}\to{B}\tens{C}$ and $\mu\colon{B}\tens{B}\to{B}$ is the multiplication map) is commutative. Moreover $\Gamma$ gives $\qs(C)$ the structure of a compact quantum semigroup. Even in the simples examples this construction produces interesting quantum semigroups.

Our aim in this paper is to show that this result is a consequence of a mild generalization of S.~Wang's work on free products of compact quantum groups (\cite{free}). Using this point of view we will answer all the questions left open by M.M.~Sadr in \cite[Section 3]{mms} about quantum semigroups obtained in this way.

We will be using the standard tools and language of the theory of compact quantum groups (\cite{cqg,bmt}) and quantum families of maps (\cite{pseu,qs}). We will also use free products of \cst-algebras (\cite{avitz} see also \cite{free}) which will always be amalgamated over the multiplies of the unit, i.e.~the units of the factor algebras will be identified. The symbol for free products will be ``$\star$'' in order to avoid confusion convolution products which are common in quantum group theory.

\section{Quantum families of maps from finite sets}

Let us consider the situation when $A$ is a unital finitely generated \cst-algebra and $B$ a commutative finite-dimensional \cst-algebra. In this case the quantum space of all maps $\qs(B)\to\qs(A)$ can be described explicitly in the following way:

\begin{theorem}\label{easy1}
Let $A$ be a unital finitely generated \cst-algebra and $B=\CC^n$ a commutative finite-dimensional \cst-algebra. Let $C$ be the \cst-algebra corresponding to the quantum space of all maps $\qs(B)\to\qs(A)$ and let
\[
\Phi\colon{A}\longrightarrow{B}\tens{C}
\]
be the quantum family of all maps $\qs(B)\to\qs(A)$. Then $C$ is isomorphic to the free product $A^{\star{n}}$ and with this isomorphism
\begin{equation}\label{Phi2}
\Phi(a)=\sum_{i=1}^ne_i\tens\iota_i(a),
\end{equation}
where $\{e_1,\dots,e_n\}$ is the standard basis of $B$ and $\iota_1,\dotsc,\iota_n$ are the natural inclusions $A\hookrightarrow{A^{\star{n}}}$.
\end{theorem}

\begin{proof}
The conclusion of the theorem may be reached by analyzing the construction of $C$ given in \cite[Theorem 3.3]{qs}. However it is much easier to simply check that $(A^{\star{n}},\Phi)$, with $\Phi$ given by \eqref{Phi2}, has the universal property of the quantum family of all maps $\qs(B)\to\qs(A)$.

This is quite easy since any unital $*$-homomorphism $\Psi\colon{A}\to{B}\tens{D}$ (for some unital \cst-algebra $D$) is of the form
\[
\Psi(a)=\sum_{i=1}^ne_i\tens\Psi_i(a),
\]
where $\Psi_1,\dotsc\Psi_n$ are unital $*$-homomorphisms $A\to{D}$. The universal property of $(A^{\star{n}},\iota_1,\dotsc,\iota_n)$ (cf.~\cite[Section 0]{avitz}) is precisely that for any collection $\Psi_1,\dots,\Psi_n$ of maps $A\to{D}$ there exists a unique $\Lambda\colon{A^{\star{n}}}\to{D}$ such that $\Lambda\comp\iota_i=\Psi_i$.
\end{proof}

\section{Free products of compact quantum semigroups}\label{fp}

In \cite{free} S.~Wang introduced so called \emph{free products} of compact quantum groups. This name is somewhat misleading since most concepts from group theory generalized to compact quantum groups do not change their meaning when applied to classical groups considered as quantum groups (described by commutative \cst-algebras). In the case of free products this is not the case. The free product of compact quantum groups which are classical is no longer a classical compact group. On the other hand, the name is well justified by considering objects dual to compact quantum groups (i.e.~discrete quantum groups, cf.~\cite[Section 3]{free}).\footnote{In the author's view the terminology ``co-free product'' could be considered as an alternative to ``free product'' in this situation.}

There is no reason why we should not apply Wang's free product construction to quantum semigroups instead of quantum groups. The proof that the resulting object is again a compact quantum semigroup is very simple (however for the main step, all be it quite trivial, S.~Wang uses quantum \emph{group} structures which are not present in our case). For the convenience of the reader we will give a brief indication of the proof of the corresponding theorem for quantum semigroups.

Let $A$ and $C$ be unital \cst-algebras. Assume that the quantum spaces $\qs(A)$ and $\qs(C)$ carry a quantum semigroup structure, i.e.~both $A$ and $B$ are equipped with coassociative morphisms $\Delta_A\colon{A}\to{A}\tens{A}$ and $\Delta_C\colon{C}\to{C}\tens{C}$ respectively. A unital $*$-homomorphism $\Theta\colon{A}\to{C}$ is a \emph{quantum semigroup morphism} if
\[
(\Theta\tens\Theta)\comp\Delta_A=\Delta_C\comp\Theta.
\]

\begin{theorem}\label{wang}
Let $\qs(A_1)$ and $\qs(A_2)$ be compact quantum semigroups and let $\Delta_k\colon{A_k}\to{A_k}\tens{A_k}$ ($k=1,2$) be the corresponding comultiplications. Let $C=A_1\star{A_2}$. Then there exists a unique $\Delta_C\colon{C}\to{C}\tens{C}$ making $\qs(C)$ a compact quantum semigroup such that the inclusion maps $\iota_{i}\colon{A_k}\hookrightarrow{C}$ ($k=1,2$) are quantum semigroup morphisms.

Moreover, for any quantum semigroup $\qs(D)$ and quantum group morphisms $\Theta_k\colon{A_k}\to{D}$ there exists a unique $\Theta_C\colon{C}\to{D}$ such that $\Theta_k=\Theta_C\comp\iota_k$ ($k=1,2$) and $\Theta_C$ is a quantum semigroup morphism.
\end{theorem}

\begin{proof}
The comultiplication $\Delta_C$ is defined uniquely by the diagram
\begin{equation}\label{DC}
\xymatrix@C=1em
{
&A_1\ar[ld]_{\Delta_2}\ar@{^{(}->}[rd]^{\iota_1}&&A_2\ar@{_{(}->}[ld]_{\iota_1}\ar[rd]^{\Delta_2}\\
A_1\tens{A_1}\ar[rrd]_{\iota_1\tens\iota_1}&&C\ar[d]^{\Delta_C}&&A_2\tens{A_2}\ar[lld]^{\iota_2\tens\iota_2}\\
&&C\tens{C}\\
}
\end{equation}
Clearly, the equality $(\Delta_C\tens\id)\comp\Delta_C(c)=(\id\tens\Delta_C)\comp\Delta_C(c)$ thus holds for $c$ in the image of either $\iota_1$ or $\iota_2$. These images generate $C$, so $\Delta_C$ is coassociative. The fact that $\iota_1$ and $\iota_2$ are quantum semigroup morphisms is built into the defining diagram \eqref{DC}.

The rest of the proof is the same as for free products of compact quantum groups (\cite[Theorem 3.4]{free}).
\end{proof}

Clearly Theorem \ref{wang} allows us to form any finite free products of quantum semigroups.

Now let us turn back to the discussion of the quantum family of all maps from a finite classical space (described by $B=\CC^n$) into a compact quantum semigroup (corresponding to a unital \cst-algebra $A$ and $\Delta_A\colon{A}\to{A}\tens{A}$). We know from Theorem \ref{easy1} that the quantum space of all maps $\qs(B)\to\qs(A)$ is $\qs(C)$, where $C\cong{A^{\star{n}}}$. By the results of Sadr discussed in Section \ref{intro} we know that there is a quantum semigroup structure on $\qs(C)$. We shall see that it coincides with the structure of the $n$-fold free product of the quantum semigroup $\qs(A)$ with itself. In the statement of the next proposition we already fix the isomorphism $C\cong{A^{\star{n}}}$, so that we may suppose that $C=A^{\star{n}}$ and the universal quantum family of maps $\Phi$ has the form \eqref{Phi2}.

\begin{proposition}
Let $\Gamma$ be the comultiplication on $C$ making the diagram \eqref{Gam} commutative and let $\Delta$ be the comultiplication on $C=A^{\star{n}}$ defined via theorem \ref{wang}. Then $\Gamma=\Delta$.
\end{proposition}

\begin{proof}
We adopt the notation introduced in Section \ref{intro}. For a simple tensor $a\tens{b}\in{A}\tens{A}$ we have
\[
(\id\tens\chi\tens\id)(\Phi\tens\Phi)(a\tens{b})=
\sum_{k,l=1}^ne_k\tens{e_l}\tens\iota_k(a)\tens\iota_l(b)=
\sum_{k,l=1}^ne_k\tens{e_l}\tens\bigl[(\iota_k\tens\iota_l)(a\tens{b})\bigr],
\]
so that for any $X\in{A}\tens{A}$ we have
\[
(\id\tens\chi\tens\id)(\Phi\tens\Phi)(X)=
\sum_{k,l=1}^ne_k\tens{e_l}\tens\bigl[(\iota_k\tens\iota_l)(X)\bigr],
\]
Thus for any $a\in{A}$
\[
(\id\tens\chi\tens\id)(\Phi\tens\Phi)\bigl(\Delta_A(a)\bigr)=
\sum_{k,l=1}^ne_k\tens{e_l}\tens\bigl[(\iota_k\tens\iota_l)\Delta_A(a)\bigr]
\]
and it follows that
\[
(\mu\tens\id\tens\id)(\id\tens\chi\tens\id)(\Phi\tens\Phi)\bigl(\Delta_A(a)\bigr)
=\sum_{k=1}^ne_k\tens\bigl[(\iota_k\tens\iota_k)\Delta_A(a)\bigr].
\]
In view of the commutativity of \eqref{Gam}, this means that
\[
(\id\tens\Gamma)\Phi(a)=\sum_{k=1}^ne_k\tens\bigl[(\iota_k\tens\iota_k)\Delta_A(a)\bigr].
\]
Recalling the formula \eqref{Phi2} we find that for $k=1,\dotsc,n$ we have $\Gamma\bigl(\iota_k(a)\bigr)=(\iota_k\tens\iota_k)\Delta_A(a)$ and this is precisely $\Delta\bigl(\iota_k(a)\bigr)$. Since the images of $\iota_1,\dotsc,\iota_n$ generate $C$, we have $\Gamma(c)=\Delta(c)$ for all $c\in{C}$.
\end{proof}

\section{Corollaries and remarks}\label{concl}

As in previous sections let $B=\CC^n$ be a commutative and finite-dimensional \cst-algebra and let $\qs(A)$ be a quantum space endowed with a structure of a quantum semigroup. Then the quantum space of all maps $\qs(B)\to\qs(A)$ corresponds to the \cst-algebra $C=A^{\star{n}}$. Moreover the natural quantum semigroup structure on $\qs(C)$ introduced by M.M.~Sadr coincides with the one of the $n$-fold free product of the quantum semigroup $\qs(A)$ with itself. We have

\begin{itemize}
\item if $\qs(A)$ is a compact quantum group then so is $\qs(C)$. Indeed, this is a special case of \cite[Theorem 3.4]{free}.
\item If $A$ admits a continuous counit then so does $C$. In fact the counit $\varepsilon\colon{A}\to\CC$ is a quantum semigroup morphism, so we can use the second part of Theorem \ref{wang}.
\item If $C$ has a continuous counit then $A$ has one by restriction.
\item If $\qs(C)$ is a compact quantum \emph{group} then so is $\qs(A)$. This is a little less immediate than the previous questions. Consider the map $\pi:C\to{A}$ defined uniquely by the fact that
\[
\iota_k\comp\pi=\id
\]
for $k=1,\dotsc,n$ (this is the map identifying all of the canonical copies of $A$ inside $C$). Then it is clear that $\pi$ is a surjective quantum semigroup morphism. Indeed, we have $(\pi\tens\pi)(\iota_k\tens\iota_k)=\id_{A\tens{A}}$, so that for $c\in{C}$ of the form $c=\iota_k(a)$ we have
\[
(\pi\tens\pi)\Delta(c)=(\pi\tens\pi)\Delta\bigl(\iota_k(a)\bigr)
(\pi\tens\pi)(\iota_k\tens\iota_k)\Delta_A(a)=\Delta_A(a)=\Delta_A\bigl(\pi(c)\bigr).
\]
The result follows for any $c\in{C}$ since $C$ is generated by images of $\iota_1,\dotsc,\iota_n$. Now the density conditions needed for $\qs(A)$ to be a quantum group are consequences of the ones for $\qs(C)$ (cf.~\cite[Definition 2.1]{cqg}).
\end{itemize}
The above statements constitute answers to all the questions left open in \cite[Section 3]{mms}

\subsection{Restricting quantum group structure}

It is tempting to prove the fact established in the last of the four statements above via the following argument. Consider a quantum space $\qs(C)$ endowed with a structure of a compact quantum group. In other words let $C$ be a unital \cst-algebra with a comultiplication $\Delta\colon{C}\to{C}\tens{C}$ such that the density conditions from \cite[Definition 2.1]{cqg} are satisfied. Now let $A$ be a unital \cst-subalgebra of $C$ (with the same unit) such that $\Delta(A)\subset{A}\tens{A}$. One could ask if $\bigl.\Delta\bigr|_A$ provides $\qs(A)$ with a structure of a compact quantum group. Using \cite[Theorem 4.2]{bmt} we immediately find that if $\qs(C)$ is \emph{coamenable} then $\qs(A)$ is indeed a compact quantum group. Moreover, one could drop the existence of a bounded counit on $C$ and prove the same using only faithfulness of the Haar measure of $\qs(C)$ (then one should use the Kustermans-Vaes non-commutative Weil theorem, cf.~\cite{vnqg}). However we have the following theorem:

\begin{theorem}
Let $\qs(C)$ be a compact quantum group with comultiplication $\Delta\colon{C}\to{C}\tens{C}$ whose Haar measure is not faithful and assume that $C$ is exact. Then there exists a unital \cst-subalgebra $A$ such that $\Delta(A)\subset{A}\tens{A}$, but $\Delta_A=\bigl.\Delta\bigr|_A$ does not make $\qs(A)$ a compact quantum group.
\end{theorem}

\begin{proof}
Let $\lambda\colon{C}\to{C_r}$ be the reducing morphism (\cite[Page 656]{pseudogr}, \cite[Section 2]{bmt}) and let $J=\ker\lambda$. Then by \cite[Proposition 4.1]{dz} (also by the proof of \cite[Theorem 2.1]{bmt}) we have $\Delta(J)\subset{J}\tens{C}$ (by exactness of $C$ we have $\ker(\lambda\tens\id)=(\ker{\lambda})\tens{C}$). In the same way we can show that $\Delta(J)\subset{C}\tens{J}$. Therefore
\[
\Delta(J)=\Delta(J^2)\subset(J\tens{C})(C\tens{J})=J\tens{J}.
\]
Let us put $A=J+\CC\I$. Then clearly $\Delta(A)\subset{A}\tens{A}$. Let us prove that $\bigl.\Delta\bigr|_A$ does not make $\qs(A)$ into a compact quantum group. Assume to the contrary that $\qs(A)$ is a compact quantum group. Then it's Haar measure is unique. But the restriction of the Haar measure of $C$ is clearly a two-sided invariant state on $A$. Now this means that the reduced version of $\qs(A)$ is one dimensional. In particular it is commutative, so the reducing morphism must be an isomorphism which is only possible when $J=\{0\}$. This contradicts the fact that the Haar measure of $\qs(C)$ is not faithful.
\end{proof}

In order to produce an example of a quantum space $\qs(C)$ with quantum group structure, non-faithful Haar measure and $C$ exact one can take the reduced $\GG_r=(A_r,\Delta_r)$ version of any non-coamenable quantum group $\GG$ such that $A_r$ is an exact \cst-algebra  (this is the case for all cocommutative examples arising from exact discrete groups). Let $\widetilde{\GG_r}$ be the compact quantum group obtained by adjoining the neutral element to $\GG_r$ (\cite[Section 8]{T}). By \cite[Proposition 8.3]{T} the corresponding \cst-algebra is the direct sum of $A_r$ and $\CC$, hence it is exact (\cite[Section 2.5.6]{Wass}). The Haar measure is not faithful by definition.

\subsection{Example}\label{example} Let us consider the special example appearing at the very beginning of the theory of quantum families of maps and in later work (\cite{pseu,qs,mms}). Let us take $B=\CC^2$ and also $A=\CC^2=\C(\ZZ_2)$. The \cst-algebra $C$ corresponding to the quantum space of all maps from a two point set to the group $\ZZ_2$ is then isomorphic to $\cst(\ZZ_2*\ZZ_2)$. Clearly the free product construction gives in this case the standard cocommutative comultiplication on the group \cst-algebra (cf.~\cite[Example 3.9(2)]{free}). In more concrete terms we can describe the \cst-algebra $C$ as the algebra of continuous functions $[0,1]\to{M_2(\CC)}$ whose values at the end-points are diagonal. In this picture $A$ is generated\footnote{Note that $A$ is nothing else than the universal unital \cst-algebra generated by two projections with no relations.} by $p$ and $q$, where for $t\in[0,1]$
\[
p(t)=
\begin{bmatrix}
0&0\\0&1
\end{bmatrix},\qquad
q(t)=\tfrac{1}{2}\begin{bmatrix}
1-\cos{2\pi{t}}&\mathrm{i}\sin{2\pi{t}}\\
-\mathrm{i}\sin{2\pi{t}}&1+\cos{2\pi{t}}
\end{bmatrix}.
\]
The comultiplication $\Delta$ is then given by
\[
\begin{split}
\Delta(p)&=(p-\I)\tens{p}+\I\tens\I+p\tens(p-\I),\\
\Delta(q)&=(q-\I)\tens{q}+\I\tens\I+q\tens(q-\I).
\end{split}
\]
The quantum space $\qs(C)$ happens to be the quantum space $\QMap\bigl(\qs(\CC^2)\bigr)$ of all maps from a two point set to itself, so it carries a natural quantum semigroup structure (cf.~Section \ref{intro}), but this is never a quantum group structure (\cite[Proposition 2.1]{ajse}). This other structure is given by the unital $*$-homomorphism $C\to{C}\tens{C}$
\[
p\longmapsto{p\tens{p}+(\I-p)\tens{q}},\qquad{q}\longmapsto{q\tens{p}+(\I-q)\tens{q}}.
\]

\section{Maps from non-classical sets}

Let $B$ be a commutative finite-dimensional algebra and let $A$ be a unital finitely generated \cst-algebra such that $\qs(A)$ is a compact quantum group. In Section \ref{fp} we reproved M.M.~Sadr's result that the quantum space of all maps $\qs(C)$ from $\qs(B)$ to $\qs(A)$ is a quantum semigroup. In the original formulation from \cite{mms} the comultiplication $\Delta\colon{C}\to{C}\tens{C}$ was introduced via the universal property of $C$ (cf.~the diagram \eqref{Gam}). Now this universal property could only be used due to the fact that the multiplication map $\mu\colon{B}\tens{B}\to{B}$ is a homomorphism. This is no longer the case if we consider $B$ non-commutative. Nevertheless, one could wonder if there is some other way of defining the quantum semigroup structure on $C$ even when $B$ is non-commutative.

In this section we will show that there are severe problems with generalizing the phenomena discussed so far to the case when $\qs(B)$ is no longer a classical finite space. More precisely we have

\begin{theorem}\label{noqg}
Let $B=M_2(\CC)$ and $A=\CC^2=\cst(\ZZ_2)$ with its standard cocommutative comultiplication. Let $\qs(C)$ be the quantum space of all maps $\qs(B)\to\qs(A)$. Then $C$ does not admit a compact quantum group structure.
\end{theorem}

Before proving Theorem \ref{noqg} let us point out, as we remarked in Section \ref{concl}, that if $\qs(A)$ has a quantum \emph{group} structure then so does $\qs(C)$. Therefore the example of Theorem \ref{noqg} means that in case of maps from a non-classical set to a compact quantum group the situation is certainly more complicated. However, it does not rule out the existence of a quantum semigroup structure on $\qs(C)$. Indeed, any quantum space is easily seen to have a quantum semigroup structure.

\begin{proof}[Proof of Theorem \ref{noqg}]
Let $\qs(C)$ be the quantum space of all maps from $\qs(B)$ to $\qs(A)$ and let $\Phi\colon{A}\to{B}\tens{C}$ be the quantum family of all these maps. Then it is not difficult to see that $C$ is isomorphic to the universal unital \cst-algebra generated by three elements $p,q$ and $z$ with the relations
\begin{align*}
&&&&p&=p^*,&p&=p^2+z^*z,&zp&=(\I-q)z,&&&&\\
&&&&q&=q^*,&q&=q^2+zz^*.
\end{align*}
The $*$-homomorphism $\Phi$ is then defined as
\[
\Phi\bigl(
\bigl[\begin{smallmatrix}
1\\0
\end{smallmatrix}
\bigr]
\bigr)=
\begin{bmatrix}
p&z^*\\z&q
\end{bmatrix}\in{M_2(\CC)}\tens{C}.
\]

Assume now that $\qs(C)$ has a quantum group structure. Then the set $X$ of characters of the \cst-algebra $C$ equipped with weak$^*$-topology carries a structure of a compact group (cf.~e.g.~\cite[Section 4.1]{ba}). We will show that $X$ is homeomorphic to the space which is the disjoint sum of the two dimensional sphere and two points. This compact space does not admit a structure of a topological group. Indeed, the neutral element cannot belong to the sphere because then the sphere would be the connected component of the identity and this is impossible by \cite[Section 3.C]{hatch}. On the other hand, if the neutral element were one of the isolated points, the group would have to be discrete, hence finite.\footnote{Another argument showing that $X$ cannot be a group is that its topology is clearly non-uniform. There is no homeomorphism of $X$ onto itself mapping any of the isolated points onto a point on the sphere.}

It is easy to see that any character of $C$ belongs to one of the following four families:
\[
\bigl\{\chi^+_\zeta\bigr\}_{|\zeta|<\frac{1}{2}},\qquad
\bigl\{\chi^-_\zeta\bigr\}_{|\zeta|<\frac{1}{2}},\qquad
\bigl\{\chi^0_\zeta\bigr\}_{|\zeta|=\frac{1}{2}},\qquad\bigl\{\omega_0,\omega_1\bigr\},
\]
where $\omega_k(p)=\omega_k(q)=k$, $\omega_k(z)=0$ for $k=0,1$ while
\begin{align*}
&&&&\chi^+_\zeta(p)&=\tfrac{1}{2}+\sqrt{\tfrac{1}{4}-|\zeta|^2},&
\chi^+_\zeta(q)&=\tfrac{1}{2}-\sqrt{\tfrac{1}{4}-|\zeta|^2},&
\chi^+_\zeta(z)&=\zeta,&&&&\\
&&&&\chi^-_\zeta(p)&=\tfrac{1}{2}-\sqrt{\tfrac{1}{4}-|\zeta|^2},&
\chi^-_\zeta(q)&=\tfrac{1}{2}+\sqrt{\tfrac{1}{4}-|\zeta|^2},&
\chi^-_\zeta(z)&=\zeta
\end{align*}
and
\[
\chi^0_\zeta(p)=\chi^0_\zeta(q)=\tfrac{1}{2},\qquad\chi^0_\zeta(z)=\zeta.
\]
It is also not difficult to see that with weak$^*$-topology the three families $\{\chi^+_\zeta\}_{|\zeta|<\frac{1}{2}}$, $\{\chi^-_\zeta\}_{|\zeta|<\frac{1}{2}}$ and $\{\chi^0_\zeta\}_{|\zeta|=\frac{1}{2}}$ form the upper hemisphere, lower hemisphere and equator of a two-sphere (since we are dealing with functionals of norm one, the weak$^*$-convergence in $C^*$ is equivalent to pointwise convergence on the unital $*$-algebra generated by $p,q$ and $z$). The functionals $\omega_0$ and $\omega_1$ are separated from this two-sphere.
\end{proof}

Let us go back to the situation when $B$ is commutative. In Section \ref{concl} we showed that if $\qs(C)$ is a compact quantum group then so is $\qs(A)$. In the derivation of this answer we used the fact that $\qs(A)$ is a quantum sub-semigroup of $\qs(C)$. The mapping $\pi$ used in that reasoning is the unique map for which the diagram
\[
\xymatrix{
A\ar[rr]^-{\Phi}\ar@{=}[d]&&B\tens{C}\ar[d]^{\id\tens\pi}\\
A\ar[rr]_-{a\longmapsto\I\tens{a}}&&B\tens{A}}
\]
is commutative. Clearly we still have this map when $B$ is not commutative, so the fact that $\qs(A)$ is a quantum sub-semigroup of $\qs(C)$ yields some restrictions on the comultiplication on $C$.

\end{document}